\documentclass[12pt]{amsart}
\usepackage{color,graphicx,url,hhline
}
\numberwithin{figure}{section}
\numberwithin{table}{section}
\newtheorem{theorem}{Theorem}[section]
\newtheorem*{thrm}{Main Theorem}
\newtheorem{lemma}[theorem]{Lemma}
\newtheorem{prop}[theorem]{Proposition}

\theoremstyle{definition}
\newtheorem{definition}[theorem]{Definition}

\newtheorem{cor}[theorem]{Corollary}

\theoremstyle{remark}
\newtheorem{remark}[theorem]{Remark}

\numberwithin{equation}{section}
\newfont{\tap}{tap scaled 650}

\def \S{{\mathbb S}}

\def \[{[ }
\def \]{] }

\def\S{{\Sigma}}

\definecolor{dgreen}{rgb}{0,0.5,0}        
\definecolor{dred}{rgb}{0.5,0,0}

\begin{document}

\title[Surfaces, quiver mutations, and quotients of Coxeter groups]{Punctured surfaces, quiver mutations,\\ and quotients of Coxeter groups}
\author{Anna Felikson}
\address{Department of Mathematical Sciences, Durham University, Upper Mountjoy Campus, Stockton Road, Durham, DH1 3LE, UK}
\email{anna.felikson@durham.ac.uk} 
\thanks{Research was supported in part by the Leverhulme Trust research grant RPG-2019-153 (PT) and by the National Science Foundation grant DMS-2100791 (MS)}

\author{Michael Shapiro}
\address{Department of Mathematics, Michigan State University, East Lansing, MI 48824, USA}
\email{mshapiro@math.msu.edu}

\author{Pavel Tumarkin}
\address{Department of Mathematical Sciences, Durham University, Upper Mountjoy Campus, Stockton Road, Durham, DH1 3LE, UK}
\email{pavel.tumarkin@durham.ac.uk}

\begin{abstract}
In~\cite{BM}, Barot and Marsh provided an explicit construction of presentation of a finite Weyl
group $W$ by any quiver mutation-equivalent to an orientation of a Dynkin diagram with Weyl group $W$. The construction was extended in~\cite{affine} to obtain presentations for all affine Coxeter groups, as well as to construct groups from triangulations of unpunctured surfaces and orbifolds, where the groups are invariant under change of triangulation and thus are presented as quotients of numerous distinct Coxeter groups. We extend the construction to include most punctured surfaces and orbifolds, providing a new invariant for almost all marked surfaces.
  \end{abstract}

\maketitle
{\small
\setcounter{tocdepth}{1}
\tableofcontents
}

\section{Introduction}
\label{intro}

Mutations of quivers were introduced by Fomin and Zelevinsky~\cite{FZ2} in context of cluster algebras.
In~\cite{FZ2}, Fomin and Zelevinsky provide a classification of cluster algebras of finite type: those cluster algebras correspond to Dynkin diagrams. In particular, skew-symmetric cluster algebras of finite type are defined by mutation classes of quivers obtained as orientations of simply-laced Dynkin diagrams.

In~\cite{BM}, Barot and Marsh provided an explicit construction of presentation of a finite Weyl
group $W$ for any quiver mutation-equivalent to an orientation of a Dynkin diagram with Weyl group $W$. In their construction, the Weyl group is presented as a quotient of a (usually infinite) Coxeter group by a normal closure of additional ({\em cycle}) relations. The result has already found some unexpected applications. It was shown in~\cite{manifolds} that the normal closure of the cycle relations is torsion-free, which was used to construct geometric manifolds with the original Weyl group acting by isometries.

The results of~\cite{BM} were extended in~\cite{affine} to obtain presentations for all affine Weyl groups, where some new types of additional relations are required. Most quivers of finite and affine type can be realized as quivers of triangulations of marked bordered surfaces. In the same paper, the construction was extended further to provide a group for every triangulation of an unpunctured marked bordered surface, where this group is invariant under change of triangulation and thus is presented as a quotient of numerous distinct Coxeter groups. 

 A construction similar to the one in~\cite{BM} to obtain presentations of Artin braid groups for any quiver of finite type (with potential) was provided by Grant and Marsh in~\cite{GM} (see also~\cite{Q1}). These results were generalized by Qiu and Zhou in~\cite{QZ}, where a finite presentation for a braid group for any unpunctured surface was constructed (and it was shown that this group coincides with the braid twist group). Taking the quotient of this braid group by setting all generators to be involutions recovers the group constructed in~\cite{affine}.   

 The results of~\cite{QZ} were generalized to punctured surfaces with boundary by Qiu~\cite{Q2}. More precisely, a finite presentation of an invariantly defined mixed twist group was given in any triangulation with all punctures contained in self-folded triangles.

\medskip

The aim of the present paper is to extend the construction from~\cite{affine} to the case of quivers arising from punctured triangulated surfaces.

More precisely, we introduce a group for every triangulation of a marked surface not containing loops, and show that the group does not depend on the triangulation. The group is a quotient of a certain Coxeter group, and, when restricted to triangulations of unpunctured surfaces without loops, coincides with the one defined in~\cite{affine} (see also~\cite{manifolds,EAWG}). The construction allows us to extend the results of~\cite{EAWG} to punctured surfaces by establishing connections with extended affine Weyl groups, this will appear in a forthcoming paper. 

By technical reasons, we restrict ourselves to punctured surfaces with at least $4$ features, where a feature is either a puncture or a boundary component, see Section~\ref{few} for the discussion of the remaining surfaces. 

The main result of the paper can be formulated as follows.
   \begin{thrm}[Theorem~\ref{invariance-th}]
Let $S$ be a marked surface with at least $4$ features, $T$ is a triangulation of $S$ not containing loops, and $Q$ is the corresponding quiver. Then the group $G(Q)$ defined in Section~\ref{construction} is an invariant of $S$, i.e. it does not depend on the triangulation. 

\end{thrm}

For surfaces with boundary, the group $G=G(Q)$ turns out to be a finite index subgroup of the quotient of the mixed twist group from~\cite{Q2}. To obtain $G$, one needs to consider a triangulation of a particular type (see Section~\ref{quotient}) and impose the relations on the mixed twist group by setting all the generators (in the presentation given in~\cite{Q2}) to be involutions, and then take a subgroup of index $2^p$, where $p$ is the number of punctures.  We note that imposing additional relations allows us to give presentation of our group in more general settings, see Section~\ref{quotient} for the details.

To prove the Main Theorem, we first show that the group $G(Q)$ is invariant under flips not creating loops (Lemma~\ref{flip_invariance}), and then show that for any two triangulations without loops there exists  a sequence of flips not creating any new loops and taking one of the triangulations to a triangulation combinatorially equivalent to the other (Proposition~\ref{trans}).  

\medskip

The paper is organized as follows. In Section~\ref{prelim} we remind necessary information about mutations of quivers and triangulations of marked bordered surfaces. In Section~\ref{group}, we construct the group $G(Q)$, prove its invariance under loop-free flips, and formulate Proposition~\ref{trans} which implies the main result of the paper. Section~\ref{closed} is devoted to the proof of Proposition~\ref{trans} for closed surfaces. In Section~\ref{bordered}, we adjust the arguments from the previous section to prove Proposition~\ref{trans} for bordered surfaces. Finally, Section~\ref{sf} contains further results, in particular, generalization of the Main Theorem to triangulated orbifolds; we also discuss properties of obtained groups.

\subsection*{Acknowledgements}
We thank the anonymous referee for pointing out the connections to results of~\cite{Q1,QZ,Q2}. A.F. and P.T. are grateful to the Max Planck Institute for Mathematics in Bonn for the financial support and excellent research environment. M.S. thanks the Department of Mathematical Sciences of Durham University for support and excellent research atmosphere during his visit in the summer of 2024.

\section{Quiver mutations and triangulated surfaces}
\label{prelim}

In this section, we recall the essential facts and definitions about quivers, their mutations, and their relation to triangulated marked surfaces.

\subsection{Quivers and mutations}
\label{qm}
A {\em quiver} is an oriented multigraph without loops and $2$-cycles. We denote by $n$ the number of vertices of a quiver $Q$, and index its vertices by $1,\dots,n$.

For every vertex $k$ of a quiver $Q$ one can define an involutive operation  $\mu_k$ called {\it mutation of $Q$ in direction $k$}. This operation produces a new quiver  denoted by $\mu_k(Q)$ which can be obtained from $Q$ in the following way (see~\cite{FZ2}): 
\begin{itemize}
\item
orientations of all arrows incident to the vertex $k$ are reversed; 
\item
for every pair of vertices $(i,j)$ such that $Q$ contains arrows directed from $i$ to $k$ and from $k$ to $j$ the number of arrows joining $i$ and $j$ changes as described in Figure~\ref{quivermut}.
\end{itemize} 

\begin{figure}[!h]
\begin{center}
\includegraphics[width=0.5\linewidth]{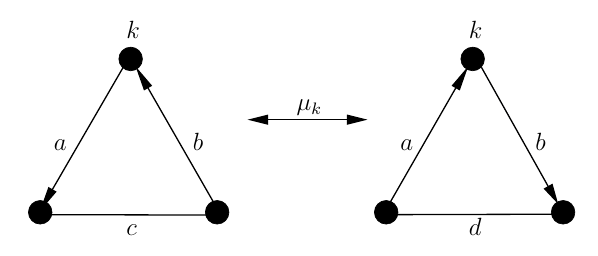}\\
\medskip
$\pm{c}\pm{d}={ab}$
\caption{Mutations of quivers. The sign before ${c}$ (resp., ${d}$) is positive if the three vertices form an oriented cycle, and negative otherwise. Either $c$ or $d$ may vanish. If $ab$ is equal to zero then neither value of $c$ nor orientation of the corresponding arrow does change.}
\label{quivermut}

\end{center}
\end{figure}

Given a quiver $Q$, its {\it mutation class} is the set of all quivers obtained from $Q$
 by all sequences of iterated mutations. All quivers from one mutation class are called {\it mutation-equivalent}.

\subsection{Quivers of finite, affine, and finite mutation types}

A quiver is of {\it finite type} if it is mutation-equivalent to an orientation of a simply-laced Dynkin diagram (i.e., of type $A,D,E$).

A quiver is of {\it affine type} if it is mutation-equivalent to an orientation of a simply-laced  affine Dynkin diagram different from an oriented cycle.

Note that any subquiver of a quiver of finite (resp. affine) type is of finite (resp., either of finite or of affine) type (see~\cite{BMR,Z}).

\begin{remark}
\label{a}
Let $\Sigma$ be a simply-laced Dynkin diagram (or a simply-laced affine Dynkin diagram different from $\widetilde A_n$). Then all orientations of $\Sigma$ are mutation-equivalent. The orientations of  $\widetilde A_{n-1}$ split into $[n/2]$ mutation classes $\widetilde A_{k,n-k}$ (each class contains a cyclic representative with only two changes of orientations with $k$ consecutive arrows in one direction and $n-k$ in the other, $0<k<n$).

\end{remark}

A quiver  is called {\it mutation-finite} (or {\it of finite mutation type}) if its mutation class is finite. Otherwise, it is called {\em mutation-infinite}.

As it is shown in~\cite{FeSTu1}, a quiver of finite mutation type
either has only two vertices, or corresponds to a triangulated surface (see Section~\ref{triang}), or belongs to one of finitely many exceptional mutation classes. In particular, all mutation-finite quivers with $3,4,5$ vertices correspond to triangulated surfaces.

\subsection{Triangulated surfaces  and quivers}
\label{triang}

The correspondence between quivers of finite mutation type and triangulated surfaces is developed in~\cite{FST}. Here we briefly remind the basic definitions.

By a {\it surface} we mean an orientable surface (possibly with boundary components) and a finite non-empty set of marked points, with at least one marked point at each boundary component. Non-boundary marked points are called {\it punctures}. 

An {\em arc} is a simple curve with endpoints at marked points (defined up to isotopy) not going through other marked points and not homotopic to a boundary segment. Two arcs are {\em compatible} if their interiors do not intersect. An (ideal) {\it triangulation} of a surface is a maximal collection of distinct mutually compatible arcs, i.e. a triangulation with vertices of triangles at the marked points. If self-folded triangles are allowed one needs to follow~\cite{FST} to consider triangulations as {\it tagged triangulations} (we are neither reproducing nor using all the details in this paper).

Given a triangulated surface, one constructs a quiver in the following way:
\begin{itemize}
\item vertices of the quiver correspond to the (non-boundary) arcs of a triangulation;
\item two vertices are connected by a simple arrow if they correspond to two sides of the same triangle (i.e., there is one simple arrow between given two vertices for every such triangle);
inside the triangle orientations of the arrow are arranged counter-clockwise (with respect to some orientation of the surface);
\item two simple arrows with opposite directions connecting the same vertices cancel out;
two simple arrows in the same direction compose a double arrow;
\item for a self-folded triangle (with two sides identified), two distinct sides of the triangle correspond to two distinct vertices of the quiver not connected by an arrow;
a vertex corresponding to the ``inner'' side of the triangle is connected to other vertices in the same way as the vertex corresponding to the outer side of the triangle.

\end{itemize}

Triangulations undergo {\em flips}: given an arc of a triangulation, one can remove it and replace with another diagonal of obtained quadrilateral (while flipping the inner side of a self-folded triangle tagged arcs should be involved, see~\cite{FST}). It was shown in~\cite{FST} that flips of triangulations correspond precisely to mutations in the corresponding vertices of the quivers.

We will extensively use the following technical result.

  \begin{lemma}
    \label{no2}
    Let $T$ be a triangulation of a marked surface $S$ containing neither loops nor self-folded triangles. Then the corresponding quiver $Q$ does not contain double arrows (unless $S$ is an annulus with precisely two marked points).
  \end{lemma}

  \begin{proof}
Suppose that there is a  double arrow in $Q$. Since $T$ contains no self-folded triangles, this implies that there are two triangles of $T$ sharing two sides to form an annulus (bounded by two loops). By the assumptions, at least one of these two loops is not a boundary arc of $S$, which proves the lemma. 

    \end{proof}

\section{Groups from triangulations}
\label{group}

In this section we define a group by a given triangulation of a punctured surface without loops. The definition follows the one from~\cite{BM}. 

\subsection{Construction}
\label{construction}

\begin{definition}
We call a triangulation {\em loop-free} if it does not contain loops and self-folded triangles.  A flip between triangulations $T$ and $T'$ is {\em loop-free} if both $T$ and $T'$ are loop-free. 

  \end{definition}

Note that by Lemma~\ref{no2} the quiver corresponding to a loop-free triangulation does not contain double arrows.
  
    Let $Q$ be a quiver of a loop-free triangulation of a marked surface $S$, index vertices of $Q$ by $1,\dots,n$. Denote by $G(Q)$ the group generated by $n$ generators $s_i$ with the following defining relations (R1)--(R3):

\begin{itemize}
\item[(R1)] $s_i^2=e$ for all $i=1,\dots,n$;

\item[(R2)] $(s_is_j)^{m_{ij}}=e$ for all $i,j$, where
$$
m_{ij}=
\begin{cases}
2 & \text{if $i$ and $j$ are not joined;} \\
3 & \text{if $i$ and $j$ are joined by an arrow.}
\end{cases}
$$

\item[(R3)] (cycle relation) for every chordless oriented cycle $C$ given by 
$$i_1{\to} i_2{\to}\cdots{\to} i_{d}{\to}i_1$$
 we take a relation
$$
(s_{i_1}s_{i_{2}}\dots s_{i_{d-1}} s_{i_{d}}s_{i_{d-1}}\dots s_{i_{2}})^2=e.
$$

\end{itemize}

\begin{remark}
  \label{cox}
  The group $G(Q)$ can be written as
  $$G(Q)=\langle s_1,\dots,s_n\mid s_i^2, (s_is_j)^{m_{ij}}, (\mathrm R3)\rangle=\tilde G(Q)/N(Q),$$
where $\tilde G(Q)$ is a Coxeter group, and $N(Q)$ is the normal closure of the relations of type (R3) in  $\tilde G(Q)$.
\end{remark}

\begin{remark}
In papers~\cite{affine,manifolds,EAWG}, where the corresponding groups were defined for quivers and diagrams from unpunctured surfaces and orbifolds (as well as for quivers and diagrams of affine type) the construction of group involves more types of relations. These types are not relevant here as they all concern subquivers corresponding to triangulated subsurfaces with loops or self-folded triangles (see~\cite{Gu1,Gu2}).
  \end{remark}

\subsection{Invariance}
\label{invariance}

In this section, we prove invariance of the group $G(Q)$ under loop-free flips, and formulate the main results.

\begin{lemma}
  \label{flip_invariance}
  Let $T$ be a loop-free triangulation of $S$, denote by $Q$ the corresponding quiver. Let $f_k$ be a loop-free flip corresponding to mutation of $Q$ at vertex $k$, denote by $Q'=\mu_k(Q)$ the quiver of triangulation $f_k(T)$. Then the group $G(Q')$ is isomorphic to $G(Q)$.

Furthermore, if we denote the generators of $G(Q')$ by $s_i'$, then the isomorphism can be constructed as follows:  $s_i'$ is mapped to  $t_i\in G(Q)$, where 
$$
t_i=\begin{cases}
s_ks_is_k & \text{if there is an arrow $i\to k$ in $Q$,} \\
s_i & \text{otherwise},
\end{cases}\eqno (*)
$$
 
\end{lemma}

\begin{remark}
\label{in-or-out}
In the definition of $t_i$ we could freely choose to conjugate $s_i$ for outgoing arrows   $i\gets k$
rather than for incoming arrows  $i\to k$: this alteration does not affect the group with generators $\{t_i\}$ since it is equivalent to conjugation of all generators by $s_k$.

\end{remark}

  \begin{proof}[Proof of Lemma~\ref{flip_invariance}]
    We need to show that for any loop-free flip the elements $t_i$ satisfy the same relations as the generators $s_i'$ of the group $G(Q')$ (and vice versa). Since $\{t_i\}_{i=1,\dots,n}$ generate  $G(Q)$, this will mean that  the groups $G(Q)$ and $G(Q')$ are isomorphic.

    For this, consider the support of any relation in $G(Q')$, there are two cases: either $k$ belongs to the support, or it does not.

    In the former case, note that the subquiver induced by the support (which is always either an oriented cycle or a single arrow) is of finite type. Therefore, the corresponding relation on $t_i$ is satisfied by~\cite{BM}.

    Thus, we may assume from now on that $k$ does not belong to the support of the relation.

    If the support consists of a single arrow $i\to j$, then $k$ can be attached to it in nine ways ($k$ is connected to each of $i$ and $j$ by at most one arrow), and it is easy to see that if the resulting subquiver is connected then it is either of finite type $A_3$ or affine type $\tilde A_2$, so  the result follows from~\cite{BM,affine}.

    Similarly, if the support consists of a triangle, we can consider all possible subquivers induced by its union with $k$. If $k$ is connected to all three vertices, then we get a mutation-infinite quiver; if $k$ is connected to one or two vertices of the triangle, we obtain a quiver of finite or affine type, so again the result follows from~\cite{BM,affine}.

    Therefore, we need to consider cycles of length at least $4$, denote such cycle by $Q'_C\subset Q'$, and its union with $k$ by $\tilde Q'_c$. By~\cite[Prop. 2.1]{S}, $k$ can be connected to $Q'_C$ by $1,2$ or $4$ arrows. 

    Assume first that $|Q_C'|=4$, then a straightforward check similar to the one above shows that the $\tilde Q'_c$ is either mutation-infinite, or satisfies one of the following conditions: either a mutation in $k$ produces a double arrow (and thus $f_k$ is not loop-free by Lemma~\ref{no2}), or $\tilde Q'_c$ is of finite or affine type (so the result follows from~\cite{affine}), or $k$ is a sink or a source (in which case we may assume that $\mu_k$ does not affect the generators, see Remark~\ref{in-or-out}).

    Finally, assume that $|Q_C'|=m\ge 5$. By~\cite{Gu1,Gu2}, the subquiver  $Q_C'$ corresponds to a once punctured disc with $m$ marked points on its boundary, each of them is joined with the puncture by an arc. Therefore, there are precisely two ways to add one more arc (corresponding to vertex $k$): either a triangle is glued to a side of the disc, or two sides of the disc are identified. In the former case $\tilde Q'_c$ is of affine type, so we are again in the assumptions of~\cite{affine}. In the latter case, the flip $f_k$ produces a loop (see Fig.~\ref{cycle}), so we come to a contradiction with the assumption of $f_k$ being loop-free.

   \end{proof}

\begin{figure}[!h]
\begin{center}
\includegraphics[width=0.5\linewidth]{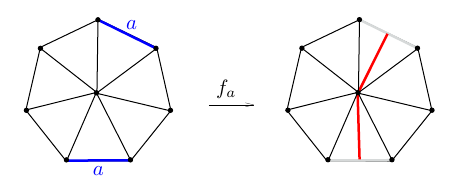}
  \caption{To the proof of Lemma~\ref{flip_invariance}: flipping the arc $a$ produces a loop.}
\label{cycle}
\end{center}
\end{figure}

   \begin{definition}
By a {\em feature} of a marked surface we mean either a puncture or a boundary component.
\end{definition}

   In Proposition~\ref{trans} below we show that all combinatorial types of loop-free triangulations can be connected by sequences of loop-free flips. As a corollary of Proposition~\ref{trans} and Lemma~\ref{flip_invariance}, we obtain the main result of the paper.

   \begin{theorem}
\label{invariance-th}
Let $S$ be a marked surface with at least $4$ features. Then the group $G(Q)$ is an invariant of $S$, i.e. it does not depend on the triangulation. 

\end{theorem}

The bulk of the paper is devoted to the proof of the following statement.

\begin{prop}
  \label{trans}
Let $S$ have at least $4$ features, and let $T$ and $T'$ be two loop-free triangulations of $S$. Then there exists a sequence of loop-free flips taking $T$ to $T''$ which is combinatorially equivalent to $T'$.

  \end{prop}

For simplicity, we first consider surfaces without boundary (Section~\ref{closed}). We then adjust the proof to the case of bordered surfaces in Section~\ref{bordered}.

  \begin{remark}
\label{counter}
Proposition~\ref{trans} does not hold if we assume that $S$ has $3$ marked points only, an example of two distinct triangulations of thrice punctured genus $2$ surface is shown in Fig.~\ref{counterexample}. The triangulations are distinct indeed: cutting the surface along all arcs connecting points $b$ and $c$ in the triangulation on Fig.~\ref{counterexample}, left, results in the identification of sides of a $(4g+2)$-gon isomorphic to the original one, while  cutting the surface along all arcs connecting points $b$ and $c$ in the triangulation on Fig.~\ref{counterexample}, right, results in an identification of sides different from the original one (more precisely, no pair of opposite sides is identified). None of the triangulations admits any loop-free flips.

   \end{remark} 

\begin{figure}[!h]
\begin{center}
\includegraphics[width=0.9\linewidth]{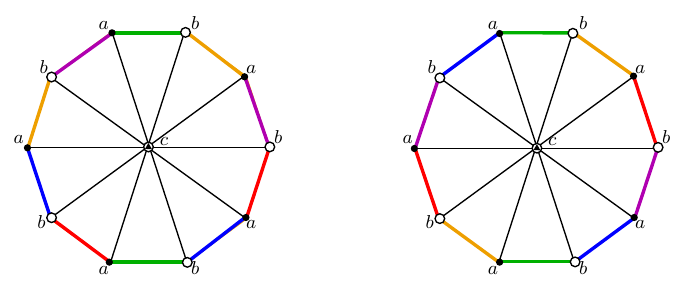}
  \caption{Two triangulations of a thrice punctured surface of genus $2$. Sides colored in the same way are identified to get an oriented surface. Symmetries of the triangulation on the left act transitively on all three vertices, while the triangulation on the right has two different types of vertices.}
\label{counterexample}
\end{center}
\end{figure}

\section{Loop-free triangulations of closed surfaces} 
\label{closed}

In this section, we assume that $S$ has no boundary and has at least three punctures. The section is devoted to proving Prop.~\ref{trans} for closed surfaces.

We proceed according to the following plan. In Section~\ref{three}, we show that there always exists a  loop-free triangulation of $S$, and describe a general form of such triangulation in the case of precisely three marked points. In Section~\ref{to-3}, we show that using loop-free flips we can transform a triangulation to a specific form, namely, if $S$ is of genus $g$ with $m$ marked points, then the resulting triangulation is a triangulation of a genus $g$ surface with $3$ marked points and additionally with $m-3$ digons glued in along one of the arcs. In Sections~\ref{glue},~\ref{four}, we prove Proposition~\ref{trans} for surfaces with precisely four marked points, and in Section~\ref{trans-closed} we extend the proof to all closed surfaces with at least four marked points.   

\subsection{Closed surfaces with three marked points}
\label{three}

\begin{lemma}
  \label{3-exists}
Let $S$ have at least $3$ marked points. Then there exists a loop-free triangulation of $S$. 
  
  \end{lemma}

  \begin{proof}
Denote by $g$ the genus of $S$.  If $S$ has precisely three punctures, then $S$ can be glued from once punctured $(4g+2)$-gon $P$ by identifying opposite sides of $P$. Triangulating $P$ by arcs connecting the puncture with all vertices provides the required triangulation, see Fig.~\ref{3} (cf. Remark~\ref{counter}). To add a puncture on $S$, place it inside a triangle and connect with all three vertices. 

    \end{proof}

\begin{figure}[!h]
\begin{center}
\includegraphics[width=0.3\linewidth]{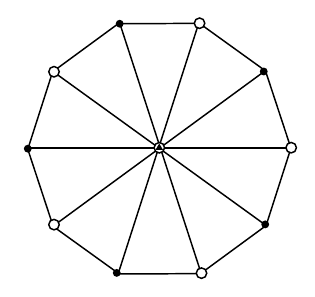}
  \caption{Any loop-free triangulation of a closed surface of genus $g$ with three punctures can be obtained by a side-pairing of a punctured $(4g+2)$-gon triangulated as shown in the figure.}
\label{3}
\end{center}
\end{figure}

\begin{remark}
  \label{all3}
  It is easy to see that all loop-free triangulations of surfaces with $3$ marked points can be constructed as in Lemma~\ref{3-exists} with different identifications of sides of the $(4g+2)$-gon.

  Note also that such a triangulation of a closed surface does not admit any loop-free flips. 
  \end{remark}

\subsection{Reducing the number of marked points to $3$}
\label{to-3}

In this section, we describe an iterative procedure of reducing the number of marked points by enclosing marked points by digons.

First, we note that Lemma~\ref{3-exists} implies that a closed surface with at least $4$ punctures always has a  loop-free triangulation.

\begin{lemma}
\label{edge}
  Let $S$ have at least $4$ marked points, and let $T$ be a loop-free triangulation. Then there exists an arc of $T$ such that both endpoints of the arc are connected to at least $3$ distinct points each.

  \end{lemma}

  \begin{proof}

Take any triangle of $T$, it has three distinct vertices. Then add iteratively neighboring triangles. By adding a neighboring triangle, we either preserve the set of vertices, or introduce a new vertex. The first time we introduce the fourth vertex, the common side of the new and old triangles satisfies the requirements of the lemma.
    \end{proof}

    In Section~\ref{closed}, by a {\em digon} of a loop-free triangulation we mean two triangles glued along two sides, see Fig.~\ref{fig-digon}; by {\em poles} of a digon we mean its vertices incident to more than two arcs.

\begin{figure}[!h]
\begin{center}
\includegraphics[width=0.08\linewidth]{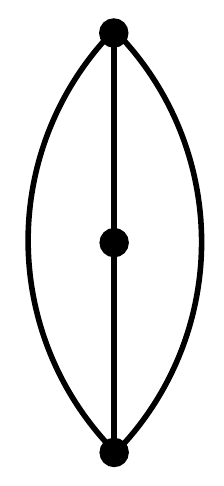}
  \caption{A digon}
\label{fig-digon}
\end{center}
\end{figure}

    \begin{lemma}
      \label{reduce}
      Let $T$ be a loop-free triangulation not containing digons, and suppose that a marked point $b$ is incident to at least four arcs of $T$, while being connected to at least three distinct marked points. Then there exists a loop-free flip in an arc incident to $b$ such that in the resulting triangulation $b$ is connected to at least three distinct marked points.    

      \end{lemma}

      \begin{proof}
        Let $a_1,\dots,a_k$ be the endpoints of arcs incident to $b$, there is a natural cyclic order on them ($a_i$ is connected to $a_{i+1}$, and $a_k$ is connected to $a_1$). Assume that some of them (say, $a_1$) appears in the list only once. Then the flip in $ba_2$ is loop-free, and $b$ is connected to distinct vertices $a_1$, $a_3$ and $a_4$.

        If none of $a_i$ appears only once, then choose $i$ such that  $a_i,a_{i+1}$ and $a_{i+2}$ are all distinct (such $i$ clearly exists), and flip in $ba_{i+1}$.
        
        \end{proof}

        \begin{remark}
          \label{rem-reduce}
          Note that after an application of the flip described in Lemma~\ref{reduce} the number of arcs incident to $b$ decreases. Also, the proof of Lemma~\ref{reduce} can be easily adjusted to find {\em two} distinct flips with the required property.

          \end{remark}

     \begin{lemma}
     \label{make_digon}
Let $S$ be a closed surface with at least four marked points, let $T$ be a loop-free triangulation of $S$. Then there exists a sequence of loop-free flips resulting in a triangulation containing a digon.

     \end{lemma}

     \begin{proof}
Assume that $T$ contains no digons. By Lemma~\ref{edge}, there is a marked point $b$ connected to at least three other marked points in $T$. If $b$ is incident to more than three arcs, then by iterated application of Lemma~\ref{reduce} we can reduce the number of arcs incident to $b$ to three. A flip in any of these three arcs produces a digon.       

\end{proof}

    \begin{definition}
      \label{associate}
      Given a surface $S$ with a loop-free triangulation $T$, we define a {\em digon removal} operation  by creating a new triangulated surface obtained by removing {\em one} digon from $T$ and identifying the sides of the removed digon.  
      
      We define an {\em associate surface} $\tilde S$ with an {\em associate triangulation} $\tilde T$ by iterated application of the digon removal operation until no digon is left (see also Remark~\ref{g0}).

\end{definition}

With every digon removal operation taking a surface $S'$ to a surface $S''$ we can associate a natural map $S'\to S''$, which sends the removed digon to its side by any continuous map (identical on the side itself).  We denote by $\pi_T$ the natural map $S\to \tilde S$ obtained as a composition of these maps.

\begin{remark}
  \label{g0}
  If $S$ is a punctured sphere, iterative applications of digon removals may result in a single segment, so we need to adjust the definition of the associate triangulation: if the applications of the digon removal operation yields a surface consisting of one digon only (with identified sides), we leave this digon intact.  
  \end{remark}

  \begin{remark}
    \label{well}
    Unless $S$ is a punctured sphere, the associate surface and triangulation are well defined, i.e. they do not depend on the order of digon removals applied. Indeed, one can note that any two digons do not overlap, so that removing two digons in any order leads to the same surface. Thus, the Diamond Lemma~\cite{N} implies the results of any two maximal sequences of digon removals will be homeomorphic.

If $S$ is a punctured sphere triangulated in a way so that there exists a sequence of digon removal operations resulting in a segment, then the Diamond Lemma assures that every maximal sequence of digon removals results in a segment. Thus, the associate triangulation in this case is always a single digon with identified sides.

    \end{remark}

    \begin{lemma}
      \label{digons}
Let $T$ be a loop-free triangulation of a closed surface $S$.
      \begin{itemize}
      \item[(a)]
        Let $\Delta$ be a triangle of $T$ with vertices $a_1,a_2,a_3$, and assume that there is a digon $D$ sharing the arc $a_1a_2$ with $\Delta$. Then there exists a sequence of loop-free flips changing $T$ inside the union of $\Delta$ and $D$ only, such that  the resulting triangulation of $D\cup\Delta$ has a digon with poles $a_2$ and $a_3$.
      \item[(b)]
The sequence of loop-free flips in (a) results in a triangulation $T'$ with the same associate triangulation $\tilde T$.

        \item[(c)]
        There exists an arc $\gamma$ of $T$ and a sequence of loop-free flips of $T$ resulting in a triangulation $T'$ containing the arc $\gamma$, such that the associate triangulations $\tilde T$ and $\tilde T'$ coincide, for any arc $\delta\ne\gamma$ of $T'$ the preimage $\pi^{-1}(\delta)$ consists of one arc of $T'$, and $\pi_{T'}^{-1}(\gamma)$ consists of a collection of digons consecutively attached to each other.

        \end{itemize}
  \end{lemma}

  \begin{proof}
    To prove (a) observe that a flip in the common arc of $\Delta$ and $D$ produces a once punctured triangle with the puncture connected to all its vertices, so the punctured triangle is symmetric with respect to its three vertices. Now flipping the arc connecting the puncture to $a_1$ creates the digon with poles $a_2$ and $a_3$, see Fig.~\ref{moving_digon}.

\begin{figure}[!h]
\begin{center}
\includegraphics[width=0.5\linewidth]{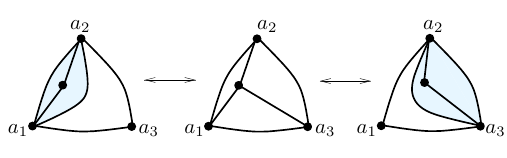}
  \caption{Moving a digon by loop-free flips.}
\label{moving_digon}
\end{center}
\end{figure}    

To prove (b), consider the sequence of two flips produced above. Removing the corresponding digon in $T$ and $T'$, we obtain the same triangulation. Since the associate triangulation is well-defined (see Remark~\ref{well}), this implies $\tilde T'=\tilde T$.

Finally, to prove (c), choose an arc $\gamma$ of $T$ as follows. Assume that $\tilde T$ is obtained from $T$ by $k$ digon removal operations $\pi_1,\dots,\pi_k$. Take any arc $\tilde\gamma$ of $\tilde T$, and consider its preimage under the last digon removal $\pi_k$. This preimage contains at most two arcs with the same endpoints as $\tilde\gamma$, choose one of them and denote it by  $\tilde\gamma_k$. Applying iteratively the same procedure $k$ times, we obtain an arc $\gamma=\tilde\gamma_1$ of $T$.

The choice of $\gamma$ guarantees that $\gamma$ is not an inner arc of any digon after applying digon removals. Now the statement of (c) is equivalent to all digons of $T'$  being consecutively attached to each other, with one of them having $\gamma$ as its side. This can be achieved by iterated application of (a).
 \end{proof}   

 \begin{remark}
  \label{standard}
As a corollary of Lemma~\ref{digons}(c), we may now assume that, after applying a sequence of loop-free flips, any loop-free triangulation $T$ can be obtained from its associate triangulation $\tilde T$ by cutting along one arc and gluing in several digons consecutively attached to each other.    
  \end{remark}

  \begin{remark}
    \label{edges}
It follows from the proof of Lemma~\ref{digons}(c) that the arc $\gamma$ in $T$ can be chosen in the preimage $\pi_T^{-1}$ of any of the arcs of $\tilde T$. In particular, the number of such arcs in $T$ is greater than one.

    \end{remark}

 \begin{lemma}
   \label{commute}

Let $S$ be a closed surface with a loop-free triangulation $T$ as in Remark~\ref{standard}, and let $\tilde T$ be the associate triangulation of $\tilde S$. Let $\tilde f$ be a loop-free flip of $\tilde T$. Then there exists a sequence of loop-free flips taking $T$ to a triangulation $T'$ such that $\tilde T'$ coincides with $\tilde f(\tilde T)$.

   \end{lemma}

   \begin{proof}
Let $\tilde \gamma$ be the arc of $\tilde T$ changed by flip $\tilde f$. Then $\pi^{-1}(\tilde \gamma)$ is either an arc of $T$ or a union of digons. If $\pi^{-1}(\tilde \gamma)$ is an arc, then the flip in this arc produces the required triangulation $T'$. Otherwise, due to Remark~\ref{edges}, we can apply  Lemma~\ref{digons}(c) to create a triangulation with the same associate triangulation $\tilde T$ but with $\pi^{-1}(\tilde \gamma)$ being a single arc.    

     \end{proof}

     \begin{lemma}
\label{ass-3}
Let $S$ be a closed surface with a loop-free triangulation $T$ as in Remark~\ref{standard}. Then there exists a sequence of loop-free flips taking $T$ to $T'$ such that $\tilde T'$ has precisely three vertices.

       \end{lemma}

       \begin{proof}
         Assume that $S$ has more than three marked points. Applying Lemma~\ref{make_digon} (if needed) and taking the associate triangulation, we obtain a loop-free triangulation of a surface $\tilde S$ with a smaller number of marked points. Lemma~\ref{commute} assures that we can continue this procedure: applying Lemma~\ref{make_digon} on $\tilde S$ can be considered as a sequence of loop-free flips on $S$. Iterated application of this procedure results in a triangulation with precisely three vertices.    
         
         \end{proof}

         \subsection{Regluing polygons}
         \label{glue}
In this section we define an operation of regluing polygons, which will be used later to change the combinatorial type of a loop-free triangulation of a surface with four marked points by loop-free flips.

         Let $P$ be a $(4g+2)$-gon with an identification of sides resulting in a twice punctured surface $S$ of genus $g$, such that every side of $P$ has two distinct endpoints. Define an {\em admissible regluing} of $P$ as the following two-step procedure.
         \begin{itemize}
         \item Cut $P$ along a diagonal of $P$ with endpoints representing distinct points in $S$.
           \item Attach the two obtained polygons along any pair of identified sides of $P$ (without changing the orientation.
           \end{itemize}
The result of the procedure is again a $(4g+2)$-gon with another identification of sides resulting in the same surface $S$.

 \begin{lemma}
               \label{regluing}
    Let $P$ be a $(4g+2)$-gon with an identification of sides resulting in a twice punctured surface of genus $g$, such that every side of $P$ has two distinct endpoints. There exists a sequence of admissible regluings resulting in a polygon in which all pairs of opposite sides are identified.            

               \end{lemma}

               \begin{proof}
                 Label the sides of $P$ by $e_1,\dots, e_{2g+1}$, where the identified sides have the same label. If we read the labels clockwise in the target polygon, we get the same word repeated twice. Now read the labels clockwise in $P$, they produce a word $w$. We prove the following claim: if $W$ contains a word $e_{i_1}\dots e_{i_k}$ and also a separate word $e_{i_2}\dots e_{i_k}$, then there exists an admissible regluing such that the resulting word contains two copies of a word of length $k$. Applying the claim repeatedly, we obtain the statement of the lemma.

                 To prove the claim, observe that the words  $e_{i_1}\dots e_{i_k}$ and $e_{i_2}\dots e_{i_k}$ divide the remaining sides of $P$ into two parts.

                 Assume first that the second side labeled $e_{i_1}$ lies between $e_{i_k}$ and $e_{i_2}$ in clockwise order (see Fig.~\ref{glue1} for an example with $k=2$), denote its endpoints by $a_{i_1}$ and $a_{i_1}'$ (in clockwise order). Similarly, denote by $a_{i_2}$ and $a_{i_2}'$ the endpoints of the closest to $a_{i_1}$ side labeled $e_{i_2}$ in clockwise order (see Fig.~\ref{glue1}).  

\begin{figure}[!h]
\begin{center}
\includegraphics[width=0.7\linewidth]{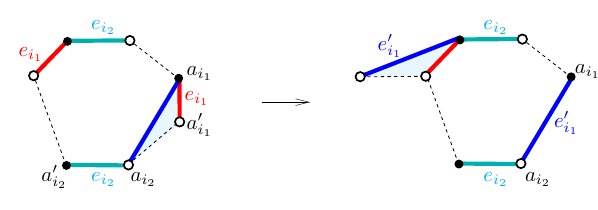}
  \caption{To the proof of Lemma~\ref{regluing}: the case of the second side labeled $e_{i_1}$ lying between $e_{i_k}$ and $e_{i_2}$ in clockwise order.}
\label{glue1}
\end{center}
\end{figure}

Note that vertices  $a_{i_1}$ and  $a_{i_2}$ represent distinct points in the surface (otherwise the arc labeled $e_{i_2}$ is a loop). Cut $P$ along the diagonal $a_{i_1}a_{i_2}$ and glue the two parts along the two sides labeled $e_{i_1}$, this is an admissible regluing. In the obtained polygon the corresponding word contains two copies of word  $e_{i_1}'e_{i_2}\dots e_{i_k}$,  see Fig.~\ref{glue1}.

Assume now that the second side labeled $e_{i_1}$ lies between $e_{i_k}$ and $e_{i_1}$ in clockwise order (see Fig.~\ref{glue2}, left, for an example with $k=2$), denote its endpoints by $a_{i_1}$ and $a_{i_1}'$ (in clockwise order). Denote by $\bar a_{i_1}$ and $\bar a_{i_1}'$ the endpoints of the {\em other} side labeled $e_{i_1}$ (in clockwise order, see Fig.~\ref{glue2}).  Vertices  $a_{i_1}$ and  $\bar a_{i_1}$ represent distinct points in the surface.

\begin{figure}[!h]
\begin{center}
\includegraphics[width=0.9\linewidth]{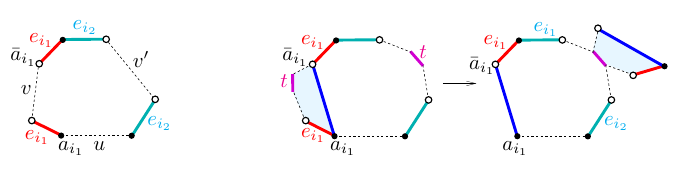}
  \caption{To the proof of Lemma~\ref{regluing}: the case of the second side labeled $e_{i_1}$ lying between $e_{i_k}$ and $e_{i_1}$ in clockwise order. Notation (left), the first case (right).}
\label{glue2}
\end{center}
\end{figure}

Denote the word (read clockwise) between $e_{i_2}\dots e_{i_k}$ and $e_{i_1}$ by $u$,   between $e_{i_1}$ and $e_{i_1}e_{i_2}\dots e_{i_k}$ by $v$, and between $e_{i_1}\dots e_{i_k}$ and $e_{i_2}\dots e_{i_k}$ by $v'$, see Fig.~\ref{glue2}. 

We now need to consider several cases.

In the first case, we assume that there exists a side labeled $t$ inside word $v$ such that its pair belongs to $v'$. Then we can do an admissible regluing by cutting $P$ along $a_{i_1}\bar a_{i_1}$ and regluing along $t$, which results in the case treated previously (as in Fig.~\ref{glue1}). 

In the second case, we assume that there exists a side labeled $t$ inside word $v$ such that its pair belongs to $u$, see Fig.~\ref{glue3}, left. Then choose any vertex $b$ inside word $v'$ representing the same point in the surface as $a_{i_1}$, and do an admissible regluing by cutting $P$ along $a_{i_1}b$ and regluing along $t$. Again, this  results in the case treated previously (as in Fig.~\ref{glue1}). 

In the last case we assume that for every side inside word $v$ its pair also belongs to $v$. Choose the side in $v$ with vertex $\bar a_{i_1}$ (let its label be $t$), denote by $b$ its second endpoint, and denote by $c$ the vertex of $P$ preceding $a_{i_1}$ clockwise. Then we can do an admissible regluing by cutting $P$ along $bc$ and regluing along $t$,  Fig.~\ref{glue3}, right. In the resulting polygon, the word between $e_{i_1}$ and $e_{i_1}e_{i_2}\dots e_{i_k}$ is a proper subword of $v$. If the polygon falls into one of the first two cases, then we proceed accordingly. Otherwise, we can repeat the procedure to shorten the word between $e_{i_1}$ and $e_{i_1}e_{i_2}\dots e_{i_k}$. Since the length of this word is always at least two, this completes the proof of the claim (and thus the lemma).

\end{proof}

\begin{figure}[!h]
\begin{center}
\includegraphics[width=0.99\linewidth]{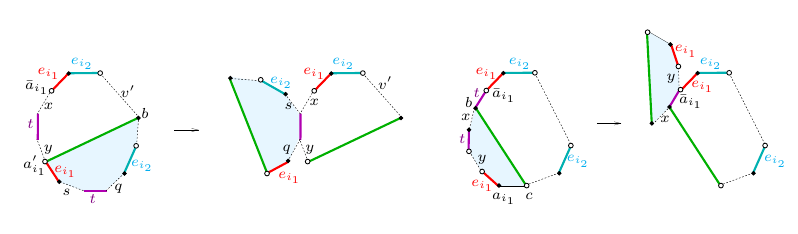}
  \caption{To the proof of Lemma~\ref{regluing}: the case of the second side labeled $e_{i_1}$ lying between $e_{i_k}$ and $e_{i_1}$ in clockwise order. The second (left) and the third (right) cases.}
\label{glue3}
\end{center}
\end{figure}

         \subsection{Surfaces with four marked points}
         \label{four}

         The aim of this section is to show that all loop-free triangulations of a closed surface with precisely four marked points are equivalent under sequences of loop-free flips.

         Let $S$ be a closed surface with precisely four marked points with any loop-free triangulation. Applying Lemma~\ref{make_digon} (if needed), we obtain a loop-free triangulation $T$ containing a digon, so its associate triangulation $\tilde T$  is described in Remark~\ref{all3}.

         \begin{lemma}
\label{double-wheel}
Let $P$ be a once punctured $(4g+2)$-gon with an identification of sides resulting in a thrice punctured surface $\tilde S$ of genus $g$, such that every side of $P$ has two distinct endpoints. Let $\tilde T$ be a triangulation of $\tilde S$ containing all sides of $P$ and all arcs connecting the puncture inside $P$ with vertices of $P$. Consider a surface $S$ with four marked points and a loop-free triangulation $T$ containing a digon, such that the associate triangulation of $T$ is precisely $\tilde T$.

Then for any two vertices $a_i$ and $a_j$ of $P$ representing distinct points of $\tilde S$ there exists a sequence of loop-free flips taking $T$ to a triangulation $T'$ of $S$ containing all sides of $P$ and an arc connecting $a_i$ and $a_j$.   
           \end{lemma}

           \begin{proof}

             Triangulation $T$ consists of triangulation $\tilde T$ with one arc substituted by a digon. By Lemma~\ref{moving_digon}, we may assume that the digon is attached to a side $\gamma$ of $P$, see Fig.~\ref{2wheel}. Denote by $b$ the puncture inside $P$ in $\tilde S$ (and $S$), and by $c$ the puncture inside the digon. The arcs $ba_i$ and $ba_j$ split $P$ into two polygons, we may assume that $c$ is contained in one of them (call it $P_1$). Performing a flip in the side of the digon distinct from the side of $P$, and then in all diagonals of $P_1$ bounding triangles with vertex $c$, we obtain a triangulation of $P_1$ in which $c$ is connected to all vertices of $P_1$, see Fig.~\ref{2wheel} for an example. Finally, applying a flip in the arc connecting $b$ and $c$ we obtain a required triangulation.

      \end{proof}          

\begin{figure}[!h]
\begin{center}
\includegraphics[width=0.99\linewidth]{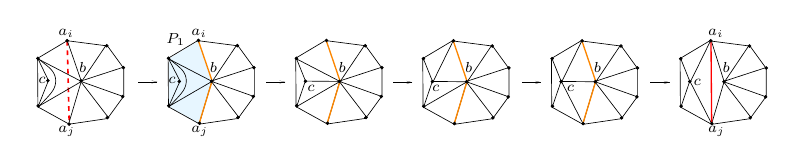}
  \caption{To the proof of Lemma~\ref{double-wheel}.}
\label{2wheel}
\end{center}
\end{figure}

We can reformulate Lemma~\ref{double-wheel} as follows.

             \begin{cor}
               \label{glue-flips}
               Let $(P_1,T_1)$ and $(P_2,T_2)$ be two pairs satisfying the assumptions of Lemma~\ref{double-wheel}, i.e. both $P_1$ and $P_2$ are $(4g+2)$-gons, but the identifications of sides are different. Assume that $P_2$ can be obtained from $P_1$ by an admissible regluing. Then $T_2$ can be obtained from $T_1$ by a sequence of loop-free flips.

               \end{cor}
             
We are now ready to prove Proposition~\ref{trans} for closed surfaces with four marked points.
             
             \begin{lemma}
               \label{trans4}
      Let $S$ be a closed surface with precisely $4$ marked points, and let $T$ and $T'$ be two loop-free triangulations of $S$. Then there exists a sequence of loop-free flips taking $T$ to $T''$ which is combinatorially equivalent to $T'$.         

               \end{lemma}
            
               \begin{proof}
                 Apply Lemma~\ref{make_digon} (if needed) to transform $T$ to a loop-free triangulation containing a digon, and consider the associate triangulation, it is represented by a punctured polygon $P$ with some identification of sides (see Remark~\ref{all3}). Our aim is to transform $P$ to a polygon with all opposite sides identified. By Lemma~\ref{regluing}, this can be done by a sequence of admissible regluings. By Corollary~\ref{glue-flips}, this sequence can be realized as a sequence of loop-free flips. Following the same procedure, we can transform $T'$ to the same shape. Now, using Lemma~\ref{digons}, we can also assume that the digons in both triangulations have poles in the same vertices of $P$, which implies that the triangulations are combinatorially equivalent.

                 \end{proof}

  \subsection{Proof of Proposition~\ref{trans} for closed surfaces}
         \label{trans-closed}                

We can now prove Proposition~\ref{trans} for closed surfaces with at least four marked points.

Consider a loop-free triangulation $T$, we may assume that the number of marked points is at least five, otherwise Lemma~\ref{trans4} applies. Applying Lemma~\ref{ass-3}, transform  $T$ to a triangulation with three vertices $a_1,a_2,a_3$ and a bunch of digons consecutively attached to each other. Choose one of the two digons attached to triangles with vertices $a_1,a_2,a_3$, and denote the marked point inside this digon by $a_4$. Apply a flip in the common side of the digon and the triangle (so that $a_4$ is now connected to all of $a_1,a_2,a_3$) to obtain a new triangulation $T_1$, and consider the associate triangulation $\tilde T_1$ with four vertices. Do the same for triangulation $T'$ to get a triangulation $\tilde T_1'$ with four vertices. Applying Lemma~\ref{trans4} to $\tilde T_1$ and $\tilde T_1'$, we can assume that they are combinatorially equivalent. Now, by Lemma~\ref{commute}, we can consider sequences of flips of  $\tilde T_1$ and $\tilde T_1'$ as sequences of loop-free flips of  $T_1$ and $T_1'$ respectively, so we transformed $T$ and $T'$ to two triangulations whose associate triangulations coincide. Now, apply Lemma~\ref{digons} to make sure the digons in both triangulations are also positioned at the same place.


\section{Loop-free triangulations of bordered surfaces}
\label{bordered}

In this section, we modify the proof of Proposition~\ref{trans} for closed surfaces presented in Section~\ref{closed} to make it work for bordered surfaces with at least $4$ {\em features} (recall that by a {\em feature} we mean either a puncture or a boundary component). We proceed mainly according to the same plan.

\subsection{New types of digons}
\label{new-digons}

The aim of this section is to show that we can treat a bordered surface $S$ as a closed surface, where some of the arcs are substituted with new types of digons as shown in Fig.~\ref{fig-digons-b}.

\begin{figure}[!h]
\begin{center}
\includegraphics[width=0.6\linewidth]{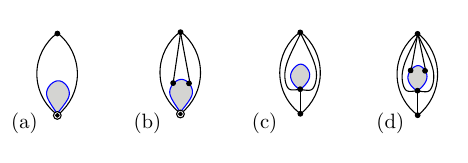}
  \caption{New types of digons: rooted digons (a),(b) and free digons (c),(d). The shaded area is a hole. The number of boundary marked points can be arbitrary. Roots of rooted digons are encircled.}
\label{fig-digons-b}
\end{center}
\end{figure}

We will call the pieces shown in Fig.~\ref{fig-digons-b}(a,b) {\em rooted digons}, and the pieces shown in Fig.~\ref{fig-digons-b}(c,d) (as well as the usual digons as in Fig.~\ref{fig-digon}) {\em free digons}. As before, the {\em poles} of a digon are the endpoints of its exterior arcs. A rooted digon has a distinguished pole which we call a {\em root}, see Fig.~\ref{fig-digons-b}. The difference in notation is caused by the properties: free digons will be treated in the same way as usual digons in closed surfaces, while rooted digons require special attention.  

\begin{lemma}
  \label{inside}
  Let $D$ be an annulus. Then for any two loop-free triangulations $T_1$ and $T_2$ of $D$ there exists a sequence of loop-free flips taking $T_1$ to a triangulation combinatorially equivalent to $T_2$.  

  \end{lemma}

  \begin{proof}
We first recall the standard terminology: an arc in $D$ is {\em bridging} if its endpoints belong to different boundary components of $D$, and {\em peripheral} otherwise. There is a partial order on peripheral arcs with endpoints on a given boundary component ($\gamma>\delta$ if cutting $D$ along $\gamma$ yields a disc containing $\delta$). Flipping a maximal peripheral arc decreases the number of peripheral arcs, so one can always transform any triangulation to one containing bridging arcs only (we call such triangulations {\em bridging triangulations}) using loop-free flips. Now, for any two bridging triangulations there is a sequence of loop-free flips taking one of them to a triangulation combinatorially equivalent to the other: a quiver of such triangulation consists of a non-oriented cycle with $m$ and $n$ arrows pointing in opposite directions (where $m,n$ are the numbers of marked points on boundary components), and all such quivers are sink-source equivalent by~\cite{CK}. As a sink-source mutation corresponds to a flip taking a bridging triangulation to another bridging triangulation, all flips are loop-free.

    \end{proof}

    \begin{remark}
      \label{egal}
Due to Lemma~\ref{inside}, we may adjust the definition of the free digon by ignoring the way the domain shown in Fig.~\ref{fig-digons-b} is triangulated (without loops).      
        
      \end{remark}

      \begin{remark}
        \label{egal-rooted}
        Similarly to Remark~\ref{egal}, for a given rooted digon we can ignore the way the corresponding domain is triangulated: all loop-free triangulations can be connected by a sequence of loop-free flips. Indeed, this is equivalent to the following: given an unpunctured  disk with at least three marked points with one chosen short diagonal $\alpha$, the set of triangulations not containing $\alpha$ is connected with respect to flips. The latter fact can be seen as follows. Denote by $b$ the marked point cut off by $\alpha$, then the set of triangulations not containing $\alpha$ is precisely the set of triangulations containing a diagonal incident to $b$. Take two triangulations from the set, let them contain diagonals $ba$ and $bc$ respectively. Then there exists a triangulation containing both $ba$ and $bc$. Now, by the result of~\cite{FST} the set of triangulations containing a fixed diagonal is connected. Therefore, both triangulations can be transformed  by flips within the set to the one containing both diagonals $ba$ and $bc$.      

        \end{remark}

\begin{lemma}
      \label{digons-b}
      Let $T$ be a loop-free triangulation of a bordered surface $S$. Let $\Delta$ be a triangle of $T$ with vertices $a_1,a_2,a_3$, and assume there is a digon $D$ sharing the arc $a_1a_2$ with $\Delta$. 
      \begin{itemize}
        \item[(a)]
      If $D$ is free or $D$ is rooted with root $a_2$, then there exists a sequence of loop-free flips changing $T$ inside the union of $\Delta$ and $D$ only, such that  the resulting triangulation of $D\cup\Delta$ has a digon with poles $a_2$ and $a_3$.
    \item[(b)]
      Let $D_1$ and $D_2$ be two digons sharing a side. Then there exists a sequence of loop-free flips
changing $T$ inside the union of $D_1$ and $D_2$ only, such that  the resulting triangulation contains $D_1$ and $D_2$ attached in opposite order (see Fig.~\ref{swap} for an example).
\end{itemize}

\end{lemma}

        \begin{proof}
          The proof of (a) is similar to the one of Lemma~\ref{digons}(a), see Figs.~\ref{move-d} and~\ref{move-b} for details.

          Statement (b) can be obtained by iterative applications of (a), see Fig.~\ref{swap}. 

        \end{proof}
        
\begin{figure}[!h]
\begin{center}
\includegraphics[width=0.52\linewidth]{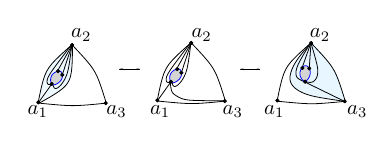}
  \caption{Moving a free digon by loop-free flips.}
\label{move-d}
\end{center}
\end{figure}

\begin{figure}[!h]
\begin{center}
\includegraphics[width=0.99\linewidth]{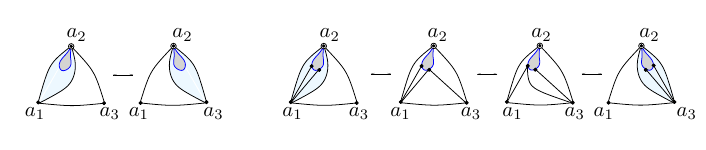}
  \caption{Moving a rooted digon by loop-free flips.}
\label{move-b}
\end{center}
\end{figure}

\begin{figure}[!h]
\begin{center}
\includegraphics[width=0.6\linewidth]{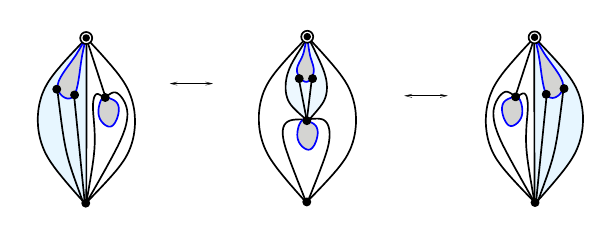}
  \caption{Swapping digons.}
\label{swap}
\end{center}
\end{figure}

\subsection{Bordered surfaces with three marked points}
\label{three-b}

\begin{lemma}
  \label{3-exists-b}
Let $S$ have at least $3$ marked points. Then there exists a loop-free triangulation of $S$. 
  
  \end{lemma}

\begin{proof}
  Consider first a surface with precisely $3$ marked points. Then we can take  a triangulation of the closed surface $S'$ of the same genus constructed in Lemma~\ref{3-exists} and amend it as follows. 
  
    Assume that some of the marked points belong to boundary components not containing other marked points. Then substituting an arc (one for each boundary component) in the triangulation of $S'$ with a rooted digon (see Fig.~\ref{3-b}) gives a required triangulation.

    If there is a boundary component with two marked points, then we can cut the triangulation of $S'$  along one arc, see Fig.~\ref{3-b}, right. Finally, if all marked points belong to the same component, then we remove one triangle from the triangulation of $S'$.

    Now, to add a marked point, glue a free digon with either a puncture or a boundary component along any arc (if the new marked point does not belong to existing boundary component), or subdivide a triangle with a boundary arc into two (otherwise).

  \end{proof}

\begin{figure}[!h]
\begin{center}
\includegraphics[width=0.6\linewidth]{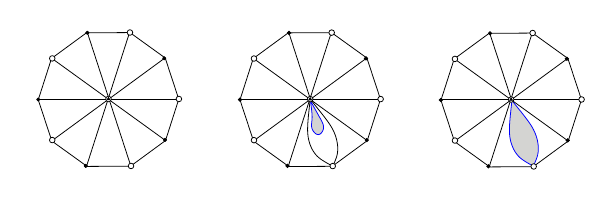}
  \caption{Constructing a loop-free triangulation of a bordered surface with three marked points from a triangulation of a closed surface.}
\label{3-b}
\end{center}
\end{figure}

As for closed surfaces, all loop-free triangulations of surfaces with $3$ marked points can be constructed as in Lemma~\ref{3-exists-b} with different identification of sides of the $(4g+2)$-gon.

\subsection{Reducing the number of features to $3$ on bordered surfaces}
\label{to-3-b}

We start with a change of triangulation in order to control the number of boundary marked points.

\begin{lemma}
  \label{boundary}
Let $T$ be a loop-free triangulation of a bordered surface $S$ with at least $4$ marked points, and assume that there is a boundary component with at least $3$ marked points. Then there exists a sequence of loop-free flips such that in the resulting triangulation at least one of the marked points of the boundary component is not incident to any non-boundary arc.     
  
  \end{lemma}

  \begin{proof}
Choose any marked point $b$ on the corresponding boundary component, denote by $c$ and $d$ the marked points on the same boundary neighboring to $b$. If $b$ is connected to at least one other marked point, proceed as in Lemma~\ref{reduce} to obtain a triangulation in which $b$ is incident to a unique non-boundary arc. Flipping this arc, we obtain a triangle with vertices $b,c,d$, in which $b$ is not incident to any non-boundary arc.  

Assume now that every arc incident to $b$ has $c$ or $d$ as its other endpoint. In this case, every triangle with vertex $b$ has sides $bc$ and $bd$, and these types of sides alternate in the neighborhood of $b$. In particular, the number of triangles with vertex $b$ is odd, so there exists a triangle $\Delta$ with vertex $b$ such that the third vertex of the triangle $\Delta'$ adjacent to $\Delta$ along the side $cd$ cannot coincide with $b$. As this third vertex cannot coincide with $c$ or $d$ either, $\Delta'$ has a vertex $f$ distinct from $b,c,d$, so after a flip of the common side of $\Delta$ and $\Delta'$ vertex $b$ is connected to $f$, and thus we are in the assumptions of the case that has been already considered.    

    \end{proof}

    \begin{remark}
\label{all-b-2}
In the assumptions that $S$ has at least two features, Lemma~\ref{boundary} allows us to restrict the considerations to surfaces with at most two marked points at every boundary component. Indeed, applying Lemma~\ref{boundary} and cutting off the newly created triangle, we can reduce the number of marked points on the chosen boundary component. We can always apply the lemma in a way that the union of the triangles that are cut off from a boundary component is connected, and thus these triangles compose a polygon, while all triangulations of a polygon are loop-free equivalent.  
      \end{remark}

Due to Remark~\ref{all-b-2}, we assume in the sequel that every boundary component has one or two marked points.

\begin{lemma}
\label{edge-b}
  Let $S$ have at least $4$ marked points, let $T$ be a loop-free triangulation without free digons, and let $b$ be a marked point. Then there exists a sequence of loop-free flips of arcs not incident to $b$ resulting in a triangulation in which $b$ is connected (by an arc or by a boundary arc) to at least $3$ distinct marked points.

  \end{lemma}

  \begin{proof}
It is easy to see that $b$ is connected to at least two marked points, so we may assume $b$ is connected to points $c$ and $d$. Take any other marked point $f$, and take a path $\psi$ from $b$ to $f$ intersecting minimal number of arcs of $T$ (note that $\psi$ intersect every arc at most once due to minimality), denote these arcs by $\gamma_1,\dots,\gamma_k$. Consider the union $C$ of all triangles intersected by $\psi$, observe that both $f$ and $b$ appear in precisely one triangle of $C$ each (again, due to minimality), so none of $\gamma_i$ has $b$ as its endpoint. Now, a flip of $\gamma_k$ produces a new arc inside $C$ incident to $f$, so it is loop-free. Continuing flipping all $\gamma_i$ till $\gamma_1$, we obtain a new arc incident to both $f$ and $b$.    

    \end{proof}
    
    \begin{lemma}
     \label{reduce-b}
     Let $S$ be a bordered surface, let $T$ be a loop-free triangulation of $S$, and let $b$ be a boundary marked point belonging to a boundary component $\beta$.
     \begin{itemize}

     \item[(a)]
       Suppose that $S$ contains at least two marked point not belonging to $\beta$. Then there exists a sequence of loop-free flips of $T$ such that the resulting triangulation contains a rooted digon with root $b$.

     \item[(b)]
       Suppose that there is a rooted digon with root $b$, and $S$ contains at least $3$ marked points not belonging to $\beta$. Then there exists a sequence of loop-free flips in arcs incident to $b$ such that the resulting triangulation contains a free digon.

     \end{itemize}
     \end{lemma}

    \begin{proof}
      In (a), we can assume that $b$ is one of the two marked points on the boundary component, otherwise there is a rooted digon already. Denote the other marked point by $c$. Denote by $a_1,\dots,a_k$ the endpoints of arcs incident to $c$ ordered clockwise (we assume that $a_i$ is connected to $a_{i+1}$, where the connecting arc might belong to boundary), $a_1=a_k=b$. By Lemma~\ref{edge-b}, we may assume that $c$ is connected to two other marked points, $d$ and $f$. If some  consecutive points $a_i$ and $a_{i+1}$ coincide, then the flip in $ca_i$ decreases the number of arcs incident to $c$ keeping the number of neighboring marked points at least $3$. Otherwise, we can decrease the number of arcs incident to $c$ proceeding as in the proof of Lemma~\ref{reduce}. Applying the procedure iteratively, we obtain a triangulation in which $c$ is incident to precisely two non-boundary arcs. Flipping any of them results in a digon with root at $b$, see Fig.~\ref{make-digons}, left.

      To prove (b), assume that there is a rooted digon $D$ with root at $b$. Substitute the rooted digon with a colored arc, and apply Lemma~\ref{edge-b} (if needed) to make sure $b$ is connected to at least three distinct marked points. Now, one can proceed as in the proof of Lemma~\ref{reduce} to decrease the number of arcs incident to $b$. If we need to flip the colored arc, then we can apply Lemma~\ref{digons-b} to color neighboring arc and make the required arc regular, after which we can flip it. Applying the procedure iteratively, we obtain a triangulation in which $b$ is incident to one colored arc and precisely two other non-boundary arcs. Substitute the colored arc with a rooted digon, and flip any of the two other arcs to obtain a free digon, see Fig.~\ref{make-digons}, right.
      
         \end{proof}

\begin{figure}[!h]
\begin{center}
\includegraphics[width=0.7\linewidth]{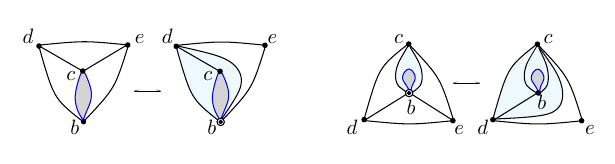}
  \caption{Constructing digons: rooted (left) and free (right).}
\label{make-digons}
\end{center}
\end{figure}

We can now define two digon removal operations similar to the one for closed surfaces (see Definition~\ref{associate}). A {\em free digon removal operation} removes from $S$ one {\em free} digon {\em with puncture} only. A {\em rooted digon removal operation} removes one rooted digon.  Both operations are well defined as corresponding digons do not overlap. The {\em associate surface} $\tilde S$ with an {\em associate triangulation} $\tilde T$ is obtained by iterated application of the two digon removal operations until no digon is left. As for closed surfaces, associate surface and triangulation are well defined (see Remark~\ref{well}). As before, we denote by $\pi_T$ the natural map $S\to \tilde S$. Also, for $g=0$ we leave one digon intact if the last digon removal is applied to a surface consisting of one digon only (cf. Remark~\ref{g0}).

\begin{lemma}
      \label{digons-b2}
Let $T$ be a loop-free triangulation of a bordered surface $S$.
      The sequences of loop-free flips in Lemma~\ref{digons-b} can be chosen in a way so that the resulting triangulation $T'$ coincides with $T$ after one or two digon removals. In particular, $T'$ has the same associate triangulation $\tilde T$.

  \end{lemma}

  \begin{proof}
The proof is identical to the proof of Lemma~\ref{digons}(b): we just require one rooted digon removal in the sequence from Lemma~\ref{digons-b}(b), and one rooted digon removal followed by one free digon removal in the sequence from Lemma~\ref{digons-b}(a) to get the same triangulations (see Figs.~\ref{move-b},~\ref{move-d}).

 \end{proof}

 \begin{lemma}
   \label{commute-b}

Let $S$ be a surface with a loop-free triangulation $T$, let $T_1$ and $S_1$ be the triangulation and the surface obtained by applying one digon removal to $T$. Let $f_1$ be a loop-free flip of $T_1$. Then there exists a sequence of loop-free flips taking $T$ to a triangulation $T'$ such that $f_1(T_1)$ can be obtained from $T'$ by at most two digon removals.

   \end{lemma}

   \begin{proof}
The proof follows the proof of Lemma~\ref{commute}: if $f_1$ flips an arc $\gamma_1$ of $T'$ that is not a side of the removed digon in $T$, then we can flip the corresponding arc $\gamma$ in $T$; otherwise we color $\gamma_1$ and use Lemma~\ref{digons-b}(a) to color neighboring arc and make  $\gamma_1$ regular, so we can now flip the corresponding arc $\gamma$ in $T$. Note that a sequence of flips in  Lemma~\ref{digons-b}(a) does not change triangulation produced after two digon removals. 

     \end{proof}

  \subsection{Proof of Proposition~\ref{trans} for bordered surfaces}
         \label{trans-bordered}

We can now prove Proposition~\ref{trans} for bordered surfaces with at least four features.

Consider a loop-free triangulation $T$. Index the boundary components $\beta_1,\dots,\beta_k$ such that the number of marked points on $\beta_i$ does not exceed the number of marked points on $b_{i-1}$. According to Remark~\ref{all-b-2}, we can apply loop-free flips to assume that every component has at most two marked points, choose a point $b_i\in\beta_i$ for all $i=1,\dots,k$.

Apply Lemma~\ref{reduce-b} to point $b_1$ to create a free digon containing boundary component $\beta_1$. By applying a rooted digon removal and then a free digon removal, we obtain a new triangulation $T_1$ of a surface $S_1$, in which the free digon is substituted with an arc. Color this arc, and apply Lemma~\ref{reduce-b} to point $b_2$. By Lemma~\ref{commute-b}, all loop-free flips on $S_1$ can be understood as sequences of loop-free flips on $S$, and if we need to flip the colored arc on $S_1$, then we can first use Lemmas~\ref{digons-b} and~\ref{digons-b2} to color a neighboring arc, and then flip the corresponding arc. If after the application of Lemma~\ref{reduce-b} to point $b_2$ the colored arc lies inside the new digon, apply Lemmas~\ref{digons-b} and~\ref{digons-b2} several times to color the side of the new digon instead.

The operation above can now be applied to further boundary components, until either we get a closed surface with at least $3$ punctures, or a bordered surface with $3$ features.

In the former case, we can consider the resulting triangulation as a triangulation of $S$ with all boundary components lying inside free digons, with all these digons attached consecutively to each other. By Lemma~\ref{digons-b} the position and the order of the digons is not relevant, and by Lemma~\ref{inside} the triangulation inside these digons is not relevant either. Applying the same procedure to triangulation $T'$ and making use of Prop.~\ref{trans} for closed surfaces (see Section~\ref{trans-closed}), we obtain required sequence of loop-free flips.

In the latter case, apply Lemma~\ref{reduce-b}(a) to create a rooted digon, and then apply rooted digon removal. Repeat the procedure (at most two times) to obtain a thrice punctured closed surface $\tilde S_0$ with triangulation $\tilde T_0$. We can consider the obtained triangulated surface as the associate surface of $S=S_0$, with triangulation $T_0$, where all free digons are consecutively attached to each other, and there are at most three rooted digons not included in free digons. As in the previous case,  the position, the order and the triangulation inside the free digons are not relevant, and by Lemma~\ref{digons-b} we can assume that any of the rooted digons have their side coinciding with any given arc incident to its root.

We now proceed as in Section~\ref{trans-closed}. Choose one free digon $D$ adjacent to a triangle not lying in any digon (we may assume it contains boundary component $\beta_1$), consider the associated triangulation $\tilde T_0$ with $D$ glued in, and apply the rooted digon removal (as before, color the resulting arc). The resulting triangulation $\tilde T_0^*$ satisfies the assumptions of Lemma~\ref{trans4}. Note that in the proof of Lemma~\ref{double-wheel} one can permute the vertices of the triangulation (see Fig.~\ref{2wheel}), so proceeding as in the proof of Lemma~\ref{trans4} we can transform $\tilde T_0^*$ to a given triangulation with prescribed boundary components (or punctures) in its four vertices. Repeating the whole procedure for triangulation $T'$, we obtain the same triangulation. Now, by the construction of $T_0$ (and the corresponding triangulation $T_0'$ for $T'$), we can assume that $T_0$ coincides with $T_0'$, which completes the proof.

\section{Further results and remarks}

\subsection{Triangulations with self-folded triangles}
\label{sf}

The main result of the paper can be reformulated for the triangulations containing self-folded triangles.  In this case, the corresponding quiver still does not contain double arrows (unless $S$ is a twice punctured disc, which we do not consider here since this is an affine type treated in~\cite{affine}).

The definition of the group $G(Q)$ should be slightly amended, we need to add new defining relations of type (R4).  

\begin{itemize}
\item[(R4)]
Let $C$ be a subquiver of $Q$ obtained from a chordless oriented cycle of length $m\ge 3$ by substituting some arrows with blocks of type IV oriented backwards, see Fig.~\ref{not-orb} for the notation and Fig.~\ref{rel-p} for an example. Denote by $1,\dots,m$ the vertices of the original cycle, and by $i_1,i_2$ the additional vertices of the block substituting the arrow $(i-1)\to i$. Define words $u_i$, $i=1,\dots,m$, as follows: either $u_i=s_{i_1}s_{i_2}s_is_{i_2}s_{i_1}$ if the arrow $(i-1)\to i$ is substituted with a block, or $u_i=s_i$ otherwise. For every such subquiver we take the relation  
$$(u_1u_2u_3\dots u_{m-1}u_{m}u_{m-1}\dots u_3u_2)^2=e$$
\end{itemize}

\begin{figure}[!h]
\begin{center}
\includegraphics[width=0.6\linewidth]{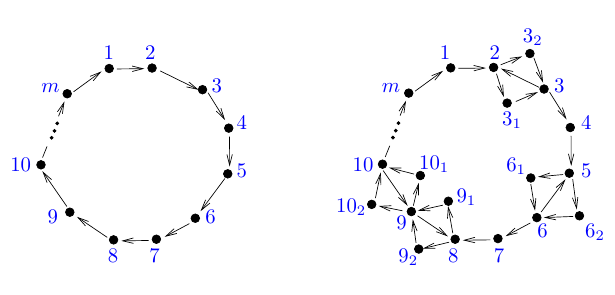}
    \caption{Support of relation of type (R4).}
\label{rel-p}
\end{center}
\end{figure} 

Note that for triangulations containing self-folded triangles the counterpart of Proposition~\ref{trans} holds: if there are two triangulations without loops, one can get rid of conjugate pairs (i.e., make the triangulation loop-free) by flipping all tagged arcs in the conjugate pairs, and then apply Proposition~\ref{trans}.

A direct verification of the defining relations (as in the proof of Lemma~\ref{flip_invariance} and in~\cite{BM,affine}) implies the following counterpart of Theorem~\ref{invariance-th}.

\begin{prop}
  \label{self}
  Let $S$ be a marked surface with at least $4$ features, let $T$ be a triangulation of $S$ not containing loops, and let $Q$ be the corresponding quiver. Then the  group $G(Q)$ with defining relations $(\mathrm{R}1)$--$(\mathrm{R}4)$ is an invariant of $S$, i.e. it does not depend on the triangulation.

  \end{prop}

\subsection{Orbifolds}
\label{orb}

The result of the previous section can be generalized further: we can consider triangulations (without loops) of orbifolds with orbifold points of order $2$, see~\cite{orbifolds}. To every such triangulation one can assign a skew-symmetrizable matrix and a quiver with valued arrows (also called {\em diagrams}). A diagram can be mutated~\cite{FZ2}, mutations correspond to flips of triangulation.

Given such a diagram $\S$ with vertices $1,\dots,n$, we can define a group $G(\Sigma)$ with generators $s_1,\dots,s_n$ and the following defining relations (R1')--(R4'):

\begin{itemize} 
\item[(R1')] $s_i^2=e$ for all $i=1,\dots,n$;

\item[(R2')] $(s_is_j)^{m_{ij}}=e$ for all $i,j$, where
$$
m_{ij}=
\begin{cases}
2 & \text{if $i$ and $j$ are not joined;} \\
3 & \text{if $i$ and $j$ are joined by a simple arrow;}\\
4 & \text{if $i$ and $j$ are joined by an arrow of weight $2$};
\end{cases}
$$

\item[(R3')] (cycle relation) for every chordless oriented cycle $C$ given by 
$$i_1{\to} i_2{\to}\cdots{\to} i_{d}{\to}i_1$$
(where some arrows may have weight $2$) we take a relation
$$
(s_{i_1}s_{i_{2}}\dots s_{i_{d-1}} s_{i_{d}}s_{i_{d-1}}\dots s_{i_{2}})^2=e.
$$

\item[(R4')]
Let $C$ be a subquiver of $Q$ obtained from a chordless oriented cycle (with simple arrows only) of length $m\ge 3$ by substituting some arrows with blocks of type IV or IV' oriented backwards, see Fig.~\ref{not-orb}. Denote by $1,\dots,m$ the vertices of the original cycle, by $i_1,i_2$ the additional vertices of block of type IV substituting the arrow $(i-1)\to i$, and by $i_0$ the additional vertex of block of type IV' substituting the arrow $(i-1)\to i$. Define words $u_i$, $i=1,\dots,m$, as follows: either $u_i=s_{i_1}s_{i_2}s_is_{i_2}s_{i_1}$ if the arrow $(i-1)\to i$ is substituted with a block of type IV, or $u_i=s_{i_0}s_is_{i_0}$ if the arrow $(i-1)\to i$ is substituted with a block of type IV', or $u_i=s_i$ otherwise. For every such subquiver we take the relation  
$$(u_1u_2u_3\dots u_{m-1}u_{m}u_{m-1}\dots u_3u_2)^2=e$$

\end{itemize}

\begin{figure}[!h]
\begin{center}
\includegraphics[width=0.6\linewidth]{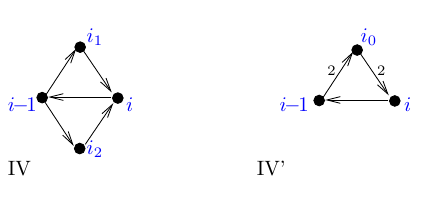}
    \caption{Blocks of types IV and IV' in the notation of~\cite{FST} and~\cite{orbifolds}}
\label{not-orb}
\end{center}
\end{figure}

A counterpart of Proposition~\ref{trans} holds in this case as well: in a triangulation of an orbifold, every orbifold point is contained in a digon, which we can treat as a free digon and proceed as in Sections~\ref{closed} and~\ref{bordered} with slight adjustments. In particular, the procedure of moving a digon shown in Fig.~\ref{moving_digon} can be realized by one flip.

Considerations similar to the ones in Section~\ref{sf} lead now to the following statement (note that by a feature we still mean a puncture or boundary component, orbifold points do not count as features).

\begin{theorem}
  \label{orb-th}
  Let $S$ be a marked orbifold with orbifold points of order $2$, suppose that $S$ has at least $4$ features. Let $T$ be a triangulation of $S$ not containing loops, and let $\S$ be the corresponding diagram. Then the  group $G(\Sigma)$ with defining relations $(\mathrm{R}1')$--$(\mathrm{R}4')$ is an invariant of $S$, i.e. it does not depend on the triangulation.

  \end{theorem}

\subsection{Moving boundary marked points between boundary components}
\label{bdry}

Given a surface or an orbifold $S$ with at least $4$ features, we can define a group $G(S)$ as a group $G(\Sigma)$ for a quiver (or diagram) of a triangulation of $S$ without loops. The following statement shows that if we move some of the boundary marked points from one boundary component to another, we get an isomorphic group. The proof is identical to the one of~\cite[Theorem 4.5]{EAWG}.

\begin{theorem}
\label{boundaries}
Let $S_{g,b,p,o,m_b}$ be a marked surface or orbifold of genus $g$  with $b$ boundary components, $p$ punctures, $o$ orbifold points, and $m_b$ boundary marked points, where $p+b\ge 4$. Then the group $G(S_{g,b,p,o,m_b})$ does not depend on the distribution of the boundary marked points along the boundary components of the surface/orbifold (depending on $g,b,p,o,m_b$ only). 

\end{theorem}

\begin{remark}
  \label{move-4}
  Note that the assumption $p+b\ge 4$ in the theorem is required since we have not proved that the group does not depend on the triangulation otherwise. However, for unpunctured surfaces and orbifolds the group was proved to be invariant in~\cite{affine}, so we can apply Theorem~\ref{boundaries} to an unpunctured surface or orbifold $S$ as soon as $S$ has two boundary components (and more than $3$ marked points). Thus, the only surfaces to which we are not able to apply Theorem~\ref{boundaries} are once punctured surfaces with precisely two boundary components (and probably some orbifold points). 

  \end{remark}

\subsection{Surfaces with a few features}
\label{few}

The results of the paper can be applied to surfaces or orbifolds with at least $4$ features. In this section we list some remarks concerning surfaces and orbifolds with fewer features.

\subsubsection{One or two features}
If the only features are punctures, then there is no loop-free triangulation, and if the surface is unpunctured then the results of~\cite{affine} apply. Thus, the only case not covered is once punctured surface/orbifold with one boundary component.

\subsubsection{Three features}
In this case there are combinatorial types of loop-free triangulations that cannot be taken one to another by loop-free flips, see Remark~\ref{counter} for an example. For closed surfaces of genus $2$ we have verified directly that the corresponding groups are isomorphic, and we expect this to hold in general.

\subsection{Relation to mixed twist groups}
\label{quotient}

For a marked surface with non-empty boundary, Qiu~\cite[Section 2.5]{Q2} defined a {\em mixed twist group} $\mathrm{MT}$ of the surface and provided its finite presentation for any (ideal) triangulation $T$ in which every puncture is contained in a self-folded triangle. If the corresponding tagged triangulation contains no loops, the presentation of $\mathrm{MT}$ is simplified to the one shown in~\cite[Corollary 2.13]{Q2}, and our group $G(Q)$ can be obtained from $\mathrm{MT}$ as follows.

The group $\mathrm{MT}$ is generated by {\em $L$-twists} (corresponding to self-folded arcs of $T$) and {\em braid twists} (corresponding to other arcs of $T$), with relations being either {\em triangle relations} or higher braid relations. First, we set all generators of $\mathrm{MT}$ to be involutions. Next, for every self-folded triangle we take the corresponding $L$-twist $a$ and braid twist $b$, and substitute the generator $a$ with $aba$ (this now corresponds to a tagged arc in the tagged triangulation), which is equivalent to taking a subgroup of index $2^p$, where $p$ is the number of punctures. It is easy to see that higher braid relations in the presentation of $\mathrm{MT}$ become Coxeter relations (R2). The triangle relations now become cycle relations (R3) by~\cite[Lemma 2.5]{GM}, so we obtain precisely the group $G(Q)$.

Since the group $\mathrm{MT}$ is defined invariantly, we can use this correspondence to extend our results.

\subsubsection{Triangulations with all punctures contained in self-folded triangles}
\label{Q}

If all punctures in a (tagged) triangulation are contained in self-folded triangles, then the corresponding quiver $Q$ does not contain oriented cycles of length greater than $3$. For such triangulations (even with loops), we can also write down the presentation of $G(Q)$ following the rules from~\cite{affine}. This can be checked explicitly, or deduced from~\cite{Q2}. Note that such triangulations exist for surfaces with non-empty boundary only.

\subsubsection{Surfaces with boundary}

The relation to the mixed twist group provides another proof of the invariance of the group $G(Q)$ for surfaces with non-empty boundary. Moreover, due to the section above, we can define the group for any marked surface with boundary, thus dropping the requirement to have at least four features.

\subsection{Example}
\label{non-iso}

In this section, we consider the group constructed in Section~\ref{group} for the case of a once punctured annulus with one marked point on each boundary component. This is the minimal example of a punctured surface with the corresponding quiver of non-finite and non-affine type. A loop-free triangulation and the corresponding quiver $Q$ are shown in Fig.~\ref{fig-ann}.

\begin{figure}[!h]
\begin{center}
\includegraphics[width=0.6\linewidth]{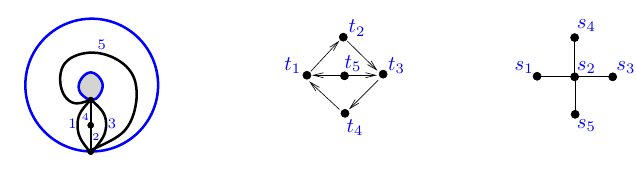}
    \caption{A loop-free triangulation of  once punctured annulus with one marked point on each boundary component (left), the corresponding quiver (middle), and standard generators of an affine Weyl group of type $\tilde D_4$ (right).}
\label{fig-ann}
\end{center}
\end{figure}

The group $G(Q)$ is then given by
$$G(Q)=\tilde G(Q)/N(Q),$$
where $\tilde G(Q)$ is a Coxeter group given by 
$$\tilde G(Q)=\langle t_1,\dots,t_5\mid t_i^2,(t_1t_2)^3,(t_1t_3)^2,(t_1t_4)^3,(t_1t_5)^3,(t_2t_3)^3,(t_2t_4)^2,(t_2t_5)^2,(t_3t_4)^3,(t_3t_5)^3,(t_4t_5)^2\rangle,$$
and $N(Q)$ is the normal closure of $(t_1t_2t_3t_4t_3t_2)^2$ in $\tilde G(Q)$, see Remark~\ref{cox}.

It is easy to check that there is a surjective homomorphism $\varphi$ of $G(Q)$ onto the affine Weyl group $W$ of type $\tilde D_4$ with generators shown in Fig.~\ref{fig-ann} (right), defined by
\begin{eqnarray*}
  t_1&\mapsto &s_1\\
  t_2&\mapsto &s_2\\
  t_3&\mapsto &s_3\\
  t_4&\mapsto &s_3s_2s_4s_2s_3\\
  t_5&\mapsto &s_4s_2s_5s_2s_4
  \end{eqnarray*}
  The verification is straightforward: one needs to check that for every defining relation for $G(Q)$ its image is trivial in $W$, which can be easily done for an affine Weyl group.

  In fact, this homomorphism is not an isomorphism: the kernel is the normal closure of the element $(t_1t_5t_3t_4t_3t_5)^2\in G(Q)$ (we use MAGMA~\cite{M} to verify this).


\begin{thebibliography}{FeLSTu2}
  
\bibitem[BM]{BM} M.~Barot, B.~R.~Marsh, {\em Reflection group presentations arising from cluster algebras}, Trans. Amer. Math. Soc. 367 (2015), 1945--1967.

\bibitem[BCP]{M} W.~Bosma, J.~Cannon, C.~Playoust, {\em The Magma algebra system. {\rm{I}}. The user language}, J. Symbolic Comput. 24 (1997), 235--265. 
  
\bibitem[BMR]{BMR} A.~B.~Buan, B.~R.~Marsh, I.~Reiten, {\it Cluster mutation via quiver representations}, Comment. Math. Helv. 83 (2008), 143--177.

  
\bibitem[CK]{CK} P.~Caldero, B.~Keller, {\em From triangulated categories to cluster algebras. II}, 
Ann. Sci. \'Ecole Norm. Sup. 39 (2006), 983--1009.

  
 \bibitem[FeLSTu]{EAWG} A.~Felikson, J.~W.~Lawson, M.~Shapiro, P.~Tumarkin,
{\em Cluster algebras from surfaces and extended affine Weyl groups}, Transform. Groups 26 (2021), 501--535.

\bibitem[FeSTu1]{FeSTu1} A.~Felikson, M.~Shapiro, P.~Tumarkin,
{\em Skew-symmetric cluster algebras of finite mutation type}, 
J. Eur. Math. Soc. 14 (2012), 1135--1180.

 \bibitem[FeSTu2]{orbifolds} A.~Felikson, M.~Shapiro, P.~Tumarkin, {\em Cluster algebras and triangulated orbifolds}, Adv. Math. 231 (2012), 2953--3002.

\bibitem[FeTu1]{affine}   A.~Felikson, P.~Tumarkin, {\em Coxeter groups and their quotients arising from  cluster algebras}, Int. Math. Res. Notices (2016), 5135--5186.

\bibitem[FeTu2]{manifolds}  A.~Felikson, P.~Tumarkin, {\em Coxeter groups, quiver mutations and geometric manifolds}, J. London Math. Soc., 94 (2016), 38--60.

  
\bibitem[FST]{FST} S.~Fomin, M.~Shapiro, D.~Thurston, {\em Cluster algebras and triangulated surfaces. Part {\rm I}: Cluster complexes}, Acta Math. 201 (2008), 83--146.


\bibitem[FZ]{FZ2} S.~Fomin, A.~Zelevinsky, {\em Cluster algebras II: Finite type classification}, Invent. Math. 154 (2003), 63--121.

\bibitem[GM]{GM} J.~Grant, B.~R.~Marsh, {\em Braid groups and quiver mutation}, Pacific J. Math. 290 (2017), 77--116.
  
\bibitem[Gu1]{Gu1} W.~Gu, {\em A decomposition algorithm for the oriented adjacency graph of the triangulations of a bordered surface with marked points},  Electron. J. Combin. 18 (2011), Paper 91, 45 pp.

 \bibitem[Gu2]{Gu2} W.~Gu, {\em Graphs with non-unique decomposition and their associated surfaces}, arXiv:1112.1008.

\bibitem[New]{N} M.~H.~A.~Newman, {\em On theories with a combinatorial definition of "equivalence"}, Ann. Math. 43 (1942), 223--243.

  
\bibitem[Sev]{S} A.~Seven, {\em Quivers of finite mutation type and skew-symmetric matrices}, Linear Algebra Appl. 433 (2010), 1154--1169.
  
\bibitem[Qy1]{Q1}
  Y.~Qiu, {\em Decorated marked surfaces: Spherical twists versus braid twists}, Math. Ann. 365 (2016), 595--633.

\bibitem[Qy2]{Q2}
Y.~Qiu, {\em Moduli spaces of quadratic differentials: Abel -- Jacobi map and deformation}, arXiv:2403.10265

\bibitem[QZ]{QZ}
Y.~Qiu, Y.~Zhou, {\em Finite presentations for spherical/braid twist groups from decorated
marked surfaces}, J. Topology 13 (2020), 501--538.
  
\bibitem[Zhu]{Z} B.~Zhu, {\em Preprojective cluster variables of acyclic cluster
algebras}, Comm. Algebra 35 (2007), 2857--2871.
    


\end{thebibliography}
\end{document}